\documentclass{amsart}
\usepackage{amsmath,amssymb, amsbsy}
\usepackage[dvips]{graphicx}
\usepackage{color}
\usepackage[latin1]{inputenc}
\usepackage[active]{srcltx}
\usepackage{wrapfig}
\usepackage{graphicx}
\usepackage[usenames,dvipsnames]{pstricks}
\usepackage{epsfig}
 \usepackage{pst-grad} % For gradients
 \usepackage{pst-plot} % For axes
\usepackage{epic}

\newcommand{\R}{\mathbb{R}}

\newcommand{\iy}{\infty}
\newcommand{\p}{\partial}

\newcommand{\Om}{\Omega}

\newcommand{\re}{\R}
\newcommand{\ren}{\R^N}

\newcommand{\dyle}{\displaystyle}
\newcommand{\ene}{{I\!\!N}}

\newcommand{\io}{\int\limits_\O}

\newcommand{\limit}{\lim\limits}

\newcommand{\dint}{\dyle\int}

\renewcommand{\a }{\alpha }
\renewcommand{\b }{\beta }
\renewcommand{\d }{\delta }
\newcommand{\D }{\Delta }
\newcommand{\e }{\varepsilon }

\newcommand{\g }{\gamma}
\newcommand{\G }{\Gamma }
\renewcommand{\l }{\lambda }

\newcommand{\n }{\nabla }

\newcommand{\s }{\sigma }

\renewcommand{\O }{\Omega }

\newtheorem{Theorem}{Theorem}[section]

\newtheorem{Definition}[Theorem]{Definition}
\newtheorem{Lemma}[Theorem]{Lemma}
\newtheorem{Proposition}[Theorem]{Proposition}
\newtheorem{remarks}[Theorem]{Remarks}
\newcommand{\cqd}{{\unskip\nobreak\hfil\penalty50
        \hskip2em\hbox{}\nobreak\hfil\mbox{\rule{1ex}{1ex} \qquad}
        \parfillskip=0pt \finalhyphendemerits=0\par\medskip}}
\begin{document}

%\lipsum
%\linenumbers

\title[Nonlinear fractional elliptic problem]{Nonlinear fractional elliptic problem with singular term at the boundary}
\author[B. Abdellaoui, K. Biroud, A. Primo]{B. Abdellaoui, K. Biroud, A. Primo}

\thanks{ This work is partially supported by project  MTM2013-40846-P and MTM2016-80474-P, MINECO, Spain.} 
\address{\hbox{\parbox{5.7in}{\medskip\noindent {B. Abdellaoui: Laboratoire d'Analyse Nonlin\'eaire et Math\'ematiques
Appliqu\'ees. \hfill \break\indent D\'epartement de
Math\'ematiques, Universit\'e Abou Bakr Belka\"{\i}d, Tlemcen,
\hfill\break\indent Tlemcen 13000, Algeria.}}}}
\address{\hbox{\parbox{5.7in}{\medskip\noindent{K. Biroud: Ecole Sup\'rieure de Management. \hfill \break\indent No. 01, Rue Barka Ahmed Bouhannak Imama,
\hfill\break\indent Tlemcen 13000, Algeria. }}}}
\address{\hbox{\parbox{5.7in}{\medskip\noindent{A. Primo: Departamento de Matem{\'a}ticas, U. Autonoma
de Madrid, \hfill\break\indent 28049 Madrid, Spain.\\[3pt]
        \em{E-mail addresses: }{\tt boumediene.abdellaoui@inv.uam.es, \tt kh$_{-}$biroud@yahoo.fr, \tt ana.primo@uam.es}.}}}}

\maketitle

\begin{abstract}
Let $\Omega\subset \mathbb{R}^N$ be a bounded regular domain, $0<s<1$ and $N>2s$.
We consider
$$
(P)\left\{
\begin{array}{rcll}
(-\D)^s u &= & \dfrac{u^{q}}{d^{2s}} &
\text{ in }\Omega , \\
u &> & 0 & \text{in }\Omega , \\
u & = & 0 & \text{in }\mathbb{R}^N \setminus\Omega ,%
\end{array}%
\right.
$$
where
$0<q\le 2^*-1$, $0<s<1$ and $d(x) = dist(x,\partial\Omega)$. {The main goal } of this paper is to analyze existence and non existence of solution to problem $(P)$ according to { the values} of $s$ and $q$.
\end{abstract}
\section{Introduction}\label{sec0}

In this paper we deal with the following problem
\begin{equation}\label{PHI}
\left\{
\begin{array}{rcll}
(-\D)^s u & = & \dfrac{u^{q}}{d^{2s}} &
\text{in }\Omega , \\
u & > & 0 & \text{in }\Omega , \\
u & = & 0 & \text{in } \ren \setminus\Omega ,%
\end{array}%
\right.
\end{equation}%
where $\O$ is a bounded regular domain in $\mathbb R^N$ ({ in a suitable sense given below}), $0<s<1$, $q\geq 0$ and $d(x) = dist(x,\partial\Omega)$. For $0<s<1$, the fractional Laplacian
$(-\Delta)^s $  is defined by
\begin{equation}\label{fraccionario}
(-\Delta)^{s}u(x):=a_{N,s}\mbox{ P.V. }\int_{\mathbb{R}^{N}}{\frac{u(x)-u(y)}{|x-y|^{N+2s}}\, dy},
\end{equation}
where
$$a_{N,s}:=2^{2s-1}\pi^{-\frac N2}\frac{\Gamma(\frac{N+2s}{2})}{|\Gamma(-s)|}$$
is the normalization constant such that the identity
$$(-\Delta)^{s}u=\mathcal{F}^{-1}(|\xi|^{2s}\mathcal{F}u),\, \xi\in\mathbb{R}^{N} $$ holds for all $u\in \mathcal{S}(\mathbb{R}^N)$
where $\mathcal{F}u$ denotes the Fourier transform of $u$ and
$\mathcal{S}(\mathbb{R}^N)$ is the Schwartz class of tempered functions.

The problem \eqref{PHI} is related to the following Hardy inequality, proved in \cite{BD}, see also \cite{FMT}  and the references therein.
More precisely, { assume that $s\in [\frac 12,1)$ and let $\Omega\subset \ren$ be a bounded domain such that the following property holds:} \\
{ $(\mathcal{R})$: there exists $x_0\in \p\O$ such that $\p\O\cap B(x_0,r)$ is ${\it \mathcal{C}^1}$, \\}
then there
exists a positive constant $C\equiv C(\O,N,s)$ such that for all $\phi\in C^{\infty}_{0}(\O)$,
\begin{equation}\label{Ter_Hardy}
\dfrac{a_{N,s}}{2}\iint_{D_\O}\dfrac{|\phi(x)-\phi(y)|^2}{|x-y|^{N+2s}}dxdy\geq C\int_{\O}\dfrac{\phi^2}{d^{2s}}dx.
\end{equation}
where $$
D_\Omega\equiv \ren \times \ren \setminus \big( \mathcal{C} \Omega \times  \mathcal{C} \Omega \big).
$$
In the case where $\O$ is a convex domain, then the constant $C$ does not depend on $\O$ and it is given by
$$
K_{N,s}= \dfrac{\Gamma^{2}(s+\frac{1}{2})}{\pi}.
$$
We refer to \cite{FMT} and the references therein for more details about the Hardy inequality.

{ In the whole paper, we will always assume that $\O$ is a $\mathcal{C}^{1,1}$ regular domain. It is clear that in this case the property $(\mathcal{R})$ holds trivially, however the $\mathcal{C}^{1,1}$ regularity is needed in order to get some precise behavior near the boundary to the auxiliary problem defined in Theorem \ref{avee}. }

\

Notice that, if $0<C(\O,N,s)<K_{N,s}$, { as it was proved in \cite{BM1} and \cite{BMS} for the local case}, it is not difficult to show that $C(\O,N,s)$ is achieved. Hence we get the existence of $\overline{u}$, a solution to the eigenvalue problem
\begin{equation}\label{EIG}
\left\{
\begin{array}{rcll}
(-\D)^s \overline{u} &= & C\dfrac{\overline{u}}{d^{2s}} &
\text{ in }\Omega , \\
\overline{u} & > & 0 & \text{in }\Omega , \\
\overline{u} & = & 0 & \text{in } \ren \setminus\Omega.
\end{array}%
\right.
\end{equation}

Assume now that $q\neq 1$, then in the local case $s=1$, problem \eqref{PHI} was considered recently in \cite{ABDM}. The authors proved a strong non existence result if $q<1$, however, for $q>1$, they proved the existence of a positive solution using suitable blow-up technics and the concentration compactness argument.
The main goal of this paper is to analyze the nonlocal case $s\in (0,1)$.

{ Notice that if the weight $d^{2s}$ is substituted by the potential weight $|x|^{-2s}$, the problem is related to the Hardy inequality proved in \cite{He}},
\begin{equation}\label{Hardylo}
\dfrac{a_{N,s}}{2}\iint_{D_\O}\dfrac{|\phi(x)-\phi(y)|^2}{|x-y|^{N+2s}}dxdy\geq \Lambda_{N,s}\,\dint_{\O} |x|^{-2s} \phi^2\,dx,\,\forall \phi\in \mathcal{C}^{\infty}_{0}(\O),
\end{equation}
where
$$
\Lambda_{N,s}= 2^{2s}\dfrac{\Gamma^2(\frac{N+2s}{4})}{\Gamma^2(\frac{N-2s}{4})}.
$$
In this case, it is not difficult to show that problem \eqref{PHI} has a non-negative solution if and only if $q<1$. We refer to \cite{AMPP} and \cite{DMPS} for more details about the related problem.

{ Let us summarize now the main results of the present paper:}

Fix $s\in (0,1)$, then if $q<1$, we are able to show the existence of solution $u$, in a suitable sense. This result makes a significative difference in comparing with the local case $s=1$, where a strong non existence result is proved. This seems to be surprising since the fractional Laplacian has less regularizing effect than the classical Laplacian.

{ Notice that a closely phenomenon occurs in the linear fractional equation as it was proved in \cite{DSV}. In that paper, the authors were able to prove that all functions are locally $s-$harmonic up to a small error. This produce more solutions in the fractional case than the local case (that disappears when letting $s\to 1$).}

{ The main result when $q<1$ is the following theorem.
\begin{Theorem}\label{intro0}
Assume that $0<s<1$, then for all $q\in (0,1)$, the problem \eqref{PHI} has a solution in a suitable sense given below, moreover $u(x)\ge C d^s(x)$ in $\O$.
\end{Theorem}
}

In the case $1<q\le 2^*$ and $s\in [\frac 12,1)$, we will show the existence of an energy solution. Taking into consideration the nonlocal nature of the operator, the proofs are more complicated than the local case, and fine computations are needed in order to get compactness results and apriori estimates.
Notice that the hypothesis $s\in [\frac 12,1)$ is needed since we will use systematically the Hardy inequality and some Liouville type results that hold for $s\ge \frac 12$. { Here we are able to show the next existence result.}

{
\begin{Theorem}\label{introd11}
Assume that $\frac{1}{2}\leq s<1$, then
\begin{enumerate}
\item If $1< q <2^{*}_{s}-1$, the problem \eqref{PHI} has a bounded positive solution  $u\in H^{s}_0(\O)\cap L^{\iy}(\O)$.
\item If $q=2^*-1$ and $\O=B_R(0)$, the problem \eqref{PHI} has a bounded radial positive solution  $u\in H^{s}_0(\O)\cap L^{\iy}(\O)$
\end{enumerate}
\end{Theorem}
}
The paper is organized as follows.
In the first section we give some auxiliary results, the concepts of solutions that we will use and some functional tools that will be needed along of  the paper.

The case $q<1$ is considered in Section \ref{sec1}. We will prove that the situation is totally different comparing with the local case. Namely, for all $s\in (0, 1)$ and for all $q\in (0,1)$, we show the existence of a distributional solution to problem \eqref{PHI}. In the case where $s<\frac12$ and under a convenient condition on $q$, we are able to prove that the solution is in a suitable fractional Sobolev space.

In Section \ref{sec2}, we treat the case $1<q\le 2^*_{s}-1$. The main idea is to combine blowing-up arguments and Liouville type theorems in order to show a priori estimates. One of the main  tools will be the Hardy inequality stated in \eqref{Ter_Hardy}. In Subsection \ref{ssec00} we treat the case $q<2^*_s-1$, then, as in \cite{ABDM}, using suitable variational arguments and Blow-up technics, we are able to prove the existence of a bounded positive solution. The critical case, $q=2^*_s-1$, is studied in Subsection \ref{sub1} under the hypothesis $\O=B_R(0)$. Then taking advantage of the radial structure of the problem, we are able to show the existence of a nontrivial radial solution.

Finally, in the last section we deal with the case $q<0$.

\section{The functional setting and tools}\label{sec1}

Let $s\in (0,1)$ and $\O\subset \ren$. We define the fractional Sobolev
space $H^{s}(\Omega)$ as
$$
H^{s}(\Omega)\equiv
\Big\{ u\in
L^2(\O):\dint_{\O}\dint_{\O}\frac{|u(x)-u(y)|^2}{|x-y|^{N+2s}}dxdy<+\infty\Big\}.
$$
$H^{s}(\O)$ is a Banach space endowed with the norm
$$
\|u\|_{H^{s}(\O)}=||u||_{L^2(\O)}+ \Big(\dfrac{a_{N,s}}{2}\dint_{\O}\dint_{\O}\frac{|u(x)-u(y)|^2}{|x-y|^{N+2s}} dx dy\Big)^{\frac
12}.
$$
Since we are working in a bounded domain, then we will use the space $H^s_0(\O)$,
$$
H^s_0(\O)=\{u\in H^s(\ren) \hbox{ with } u=0 \hbox{  a.e. \: in   } \ren\backslash \Omega \}
$$
endowed with the norm
$$
||u||^{2}_{H^s_0(\O)}=\dfrac{a_{N,s}}{2}\iint_{D_\O}\dfrac{(u(x)-u(y))^2}{|x-y|^{N+2s}}dxdy,
$$
where $D_\O=\mathbb{R}^{2N}\backslash (\O \times \O)$. It is clear that $(H^s_0(\O), ||.||_{H^s_0(\O)})$ is a Hilbert space.
We refer to \cite{DPV} and \cite{Adams} for more properties of the previous spaces.

The next Sobolev inequality is proved in \cite{DPV}, see also \cite{Ponce} for a simple proof.
\begin{Theorem} \label{Sobolev}
Assume that $0<s<1$ with $2s<N$. There exists a positive constant $S\equiv S(N,s)$ such that for all
$u\in C_{0}^{\infty}(\ren)$, we have
$$
S
\Big(\dint_{\mathbb{R}^{N}}|u(x)|^{2_{s}^{*}}dx\Big)^{\frac{2}{2^{*}_{s}}}\le
\dfrac{a_{N,s}}{2}\iint_{\re^{2N}}
\dfrac{|u(x)-u(y)|^{2}}{|x-y|^{N+2s}}\,dxdy
$$
with $2^{*}_{s}= \dfrac{2N}{N-2s}$.
\end{Theorem}
Let us begin by stating the sense in  which the solution is defined. Since we are looking for solutions to \eqref{PHI} with right hand side in $L^1(\O)$, then we will use systematically the next definition.

\begin{Definition}\label{defph4}
Assume that $h\in L^1(\O)$. We say that $u\in L^1(\O)$ is a weak solution to problem
\begin{equation}\label{hho}
\left\{
\begin{array}{rcll}
(-\D)^{s} u &= & h & \text{ in }\Omega , \\
u & = & 0 & \text{  in }\ren\backslash \Omega , \\
\end{array}%
\right.
\end{equation}
if $u=0$ in $\ren\backslash \Omega$ and for all $\psi\in \mathbb{X}_s$, we have
$$
\io u((-\D)^{s}\psi) dx =\io h\psi dx.
$$
where
$$
\mathbb{X}_s\equiv \Big\{\psi\in \mathcal{C}(\ren)\,|\,\text{supp}(\psi)\subset \overline{\O},\,\, (-\Delta )^s\psi(x) \hbox{ pointwise defined and   } |(-\Delta )^s\psi(x)|<C \hbox{ in  } \O\Big\}.
$$
\end{Definition}
In the same way we define the sense of distributional solution to \eqref{hho}.
\begin{Definition}\label{distri}
Let $h\in L^1_{loc}(\Omega)$, we say that $u\in L^1(\O)$ is a distributional solution to problem \eqref{hho} if
for all $\psi\in \mathcal{C}^\infty_0(\O)$, we have
$$
\io u((-\D)^{s}\psi) dx =\io h\psi dx.
$$
\end{Definition}

The next existence result is proved in \cite{LPPS}, \cite{CV1} and \cite{AAB}.
\begin{Theorem}\label{entropi}
Assume that $h\in L^1(\O)$, then problem \eqref{hho}  has a unique weak  solution $u$  that is obtained as the limit of $\{u_n\}_{n\in \mathbb{N}}$, the
sequence of the unique solutions to the approximating problems
\begin{equation}\label{proOO}
\left\{\begin{array}{rcll} (-\Delta)^s u_n &= & h_n(x) & \mbox{  in  }\O,\\ u_n &= & 0 & \mbox{ in } \ren\backslash\O,
\end{array}
\right.
\end{equation}
with $h_n=T_n(h)$ and $T_n(\s)=\max(-n, \min(n,\s))$. Moreover,
\begin{equation} \label{tku}
T_k(u_n)\to T_k(u)\hbox{  strongly in }   H^{s}_{0}(\Omega), \quad \forall k > 0,
\end{equation}
\begin{equation} \label{L1u}
u \in L^\theta(\Omega) \,, \qquad  \forall  \ \theta\in \big[1, \frac{N}{N-2s}\big)\,
\end{equation}
 and
\begin{equation}\label{L1du}
\big|(-\Delta)^{\frac{s}{2}}   u\big| \in L^r(\Omega) \,, \qquad \forall  \  r \in \big[1,  \frac{N}{N-s} \big) \,.
\end{equation}
Furthermore
\begin{equation}\label{L1duu}
u_n\to u\mbox{  strongly in } W^{s,q_1}_0(\O)\mbox{  for all   }q_1<\frac{N}{N-2s+1},
\end{equation}
and $u$ is an entropy solution to problem \eqref{hho} in the sense that
\begin{equation}\label{entro001}
\iint_{R_\rho}\frac{|u(x)-u(y)|}{|x-y|^{N+2s}}dxdy\to 0\mbox{   as   }\rho\to \infty,
\end{equation}
where
\begin{equation*}
R_\rho=\bigg\{(x,y)\in \re ^{2N}: \rho+1\le \max\{|u(x)|,|u(y)| \}\mbox{  with  } \min\{|u(x)|,|u(y)| \}\le \rho\mbox{  or  }u(x)u(y)<0\bigg\}
,\end{equation*}
and for all $k>0$ and $\varphi\in H^s_0(\O) \cap L^{\infty}(\O)$, we
have
\begin{equation}\label{eq:alcala}
\begin{array}{lll}
&\dyle \frac 12\iint_{D_\O}
\frac{(u(x)-u(y))[T_k(u(x)-\varphi(x))-T_k(u(y)-\varphi(y))]}{|x-y|^{N+2s}}dxdy\le \\
&\dyle \io h(x)T_k(u(x)-\varphi(x)) \, dx.
\end{array}
\end{equation}

\end{Theorem}
As a consequence of Picone inequality to the fractional operator, see \cite{LPPS}, we have the next comparison principle proved in \cite{LPPS} that extends the one obtained by Brezis-Kamin in \cite{BK} for the local case.
\begin{Lemma}\label{compar}
Assume that $0<s<1$ and let $f(x,t)$ be a Caratheodory function such that $\dfrac{f(x,t)}{t}$ is decreasing for $t>0$. Suppose $u,v\in H^{s}_{0}(\O)$ are such that
\begin{equation}\label{zzz}
\left\{
\begin{array}{cc}
(-\D)^{s} u \geq f(x,u) & \text{ in }\Omega , \\
(-\D)^{s} v \leq f(x,v) & \text{ in }\Omega . \\
\end{array}%
\right.
\end{equation}
Then $u\geq v$ a.e. in $\O$.
\end{Lemma}

In order to prove a priori estimates for approximating problem, we will use the next existence result obtained in \cite{ABBA}.
\begin{Theorem}\label{avee}
Assume that $s\in (0,1)$ and $\beta\in (0,s+1)$. Let $\O$ be a bounded regular domain { ($\mathcal{C}^{1,1}$ is sufficient)}, then the problem
\begin{equation}\label{ave1}
\left\{
\begin{array}{rcll}
(-\D)^{s} \phi &= & \dfrac{1}{d^\beta(x)} & \text{ in }\Omega , \\
\phi & = & 0 & \text{  in }\ren\backslash \Omega , \\
\end{array}%
\right.
\end{equation}
has a distributional solution such that
\begin{enumerate}
\item[A)] if $\beta<s$, then $\phi\backsimeq d^s$,
\item [B)]if $\beta=s$, then $\phi\backsimeq d^s\log (\frac{D}{d(x)})$,
\item [C)] if $\beta\in (s,s+1)$, then $\phi\backsimeq d^{2s-\beta}$.
\end{enumerate}
\end{Theorem}

\

{ To deal with the sub-critical case $1<q<2^*_{s}-1$ in problem \eqref{PHI}, we will use suitable Blow-up arguments in order to prove a priori estimates and, as a consequence, we need a Liouville type results for the fractional Laplacian. }

{ We begin by the next result obtained in \cite{JYX}.}

{
\begin{Theorem}\label{blo}
Let  $0<s<1$ and $N>2s$. Suppose that $q<2^*_s-1$, then the problem
\begin{equation*}
(-\D)^s u=u^q, \quad u>0 \quad \hbox{in} \quad\ren,
\end{equation*}%
has no locally bounded solution.
\end{Theorem}
}

{ Consider now the half space
$$
\mathbb{R}^{N}_{+}=\{z=(x',x_N)| x'\in\mathbb{R}^{N-1},x_N>0\},
$$
then the next non existence result is proved in \cite{MMF}.}

{
\begin{Theorem}\label{blo2}
Let  $0<s<1$ and $N>2s$. Suppose that $q<\frac{N-1+2s}{N-1-2s}$, then the problem
\begin{equation*}
\left\{
\begin{array}{rcll}
(-\D)^s u &= & u^q & \text{ in }\mathbb{R}^{N}_{+}, \\
u & = & 0 &\text { in }\mathbb{R}^{N}\backslash \mathbb{R}^{N}_{+},
\end{array}%
\right.
\end{equation*}
has no positive bounded solution.
\end{Theorem}
It is clear that if $q<2^*-1$, then $q<\frac{N-1+2s}{N-1-2s}$.}

\

\begin{remarks}\label{rem11}
Motivated by the work of Caffarelli and Silvestre \cite{CS1}, several authors have considered the spectral fractional Laplacian operator in a bounded domain with zero Dirichlet boundary data by means of an auxiliary variable.

More precisely, let us begin by introducing the following space
\begin{equation}\label{space}
\widetilde{H^s}(\O)=\{u=\Sigma_{i=1}^\infty a_i \phi_i \in L^{2}(\O)\text{  with } \Sigma_{i=1}^{\infty}\l^s_{i} a^{2}_i<\infty \},
\end{equation}
were $(\l_{i},\phi_i)$ are the eigenvalues and the eigenfunctions of $(-\D)$ with Dirichlet boundary conditions. $\widetilde{H^s}(\O)$ is endowed with the norm
$$
||u||^2_{\widetilde{H^s}(\O)}=\Sigma_{i=1}^\infty \l^s_{i} a^{2}_i.
$$
The space defined in \eqref{space} is the interpolation space $(H^1_{0}(\O); L^{2}(\O))_{[1-s]}$ see \cite{Adams,BRS, LM}.
Therefore we obtain that
$$\widetilde{H^s}(\O)=\left\{
\begin{array}{rcll}
H^{s}(\Omega) & \text{if } & 0<s<\frac 12 , \\
H^{\frac 12}_{00}(\Omega)& \text{if }&   s=\frac 12 , \\
H^{s}_0(\Omega)& \text{if } &  \frac 12<s<1,%
\end{array}%
\right.$$
where
$$
H^{\frac 12}_{00}(\Omega)=\left\{u\in H^{\frac 12}(\Omega)\mbox{  such that }\io \dfrac{u^2}{d(x)}<\infty \right\},
$$
we refer to \cite{LM} for more details about these spaces.

\

As consequence we define the spectral fractional Laplacian $\mathcal{A}^s$ by setting
\begin{equation}\label{space1}
\mathcal{A}^s(u)=\Sigma_{i=1}^\infty \l^s_{i} a_i \phi_i.
\end{equation}
From \cite{CS1}, the problem
\begin{equation}\label{spect1}
\left\{
\begin{array}{rcll}
\mathcal{A}^s u & = & f(x,u) &
\text{in }\Omega , \\
u & = & 0 & \text{on } \p \Omega ,%
\end{array}%
\right.
\end{equation}%
can be formulate in a local setting.

More precisely, let $C_{\O}=\O\times (0,\infty)\subset {\mathbb{R}}_{+}^{N+1}$, then a point in $C_{\O}$ will be denoted by $(x,y)$. If $s\in (\frac 12, 1)$, for $u\in H^{s}_{0}(\O) $, we define the $s$-harmonic extension $w=E_{s}(u)$ in $C_{\O}$ as the solution to the problem
\begin{equation}\label{PH2}
\left\{
\begin{array}{cc}
-\texttt{div} (y^{1-2s} \nabla w)=0 & \text{in } C_{\O} , \\
w=0 & \text{in }\p_L\O , \\
w= u & \text{on }\O\times \{0\},%
\end{array}%
\right.
\end{equation}%
with $\p_L\O=\p\O\times(0,\infty)$. It is obvious that $w$ belongs to the space
$$
X^{s}_{0}(C_\O)=\overline{\mathcal{C}^{\infty}_{0}(C_\O)}^{||.||_{X^{s}_{0}(C_\O)}},     \hbox{        with    }  ||w||_{X^{s}_{0}(C_{\O})}=\left(\int_{C_{\O}}y^{1-2s}|\n w|^2dxdy\right)^{\frac{1}{2}},
$$
where $k_s$ is a normalization constant. As consequence, for all $u\in H^{s}_{0}(\O)$, we have
\begin{equation}\label{norm}
||E_{s}(u)||_{X^{s}_{0}(C_{\O})}=||u||_{H^{s}_{0}(\O)}.
\end{equation}
Now, going back to problem \eqref{PH2}, we get
\begin{equation}\label{EXT}
\dfrac{\p w(x,y)}{\p \nu^s}\equiv -\dfrac{1}{k_s}\lim_{y\rightarrow 0^{+}}\dfrac{\p w(x,y)}{\p y}=\mathcal{A}^s u(x).
\end{equation}
Now, form \cite{CCPS}, we can prove that the non existence result in Theorem \ref{blo} is equivalent to the next one.
\begin{Theorem}\label{blow1}
Assume that $\frac 12\leq s<1$ and $N>2s$. Then the problem
\begin{equation*}
\left\{
\begin{array}{rcll}
-{\it div} (y^{1-2s} \n v) & = & 0 & \text{in } \mathbb{R}^{N+1}_{+}\equiv \ren\times (0,\infty), \\
\dfrac{\p v}{\p \nu^s} &= & Cv^{q}(x,0) & \text{on }\p\mathbb{R}^{N+1}_{+}=\ren,%
\end{array}%
\right.
\end{equation*}%
has no bounded positive solution provided that $q<2^*_s-1$.
\end{Theorem}
\end{remarks}

\

\section{The sublinear case: $0< q <1$}\label{sec2}
In this section we are interested to find a positive solution to problem \eqref{PHI} for $0< q <1$, more precisely we have the following result.
\begin{Theorem}\label{tt1}
Assume that $0<q<1$ and $0<s<1$, then problem \eqref{PHI}  has a distributional positive solution in the sense of Definition \ref{distri} such that
$u(x)\ge C d^s(x)$ in $\O$.
\end{Theorem}
\begin{proof}
We divide the proof into two main cases according to the values of $s$.

\textbf{ The first case: $0<s<\frac 12$}

We proceed by approximation. Let $u_n$ be the unique positive solution to
\begin{equation}\label{PPP1}
\left\{
\begin{array}{rcll}
(-\Delta)^{s} u_n & = & \dfrac{u^q_n}{(d(x)+\frac 1n)^{2s}} & \text{ in }\Omega , \\
u_n & > & 0 & \text{in }\Omega , \\
u_n & = & 0 & \text{in }\ren \setminus\Omega.%
\end{array}%
\right.
\end{equation}
The existence of $u_n$ follows using classical variational argument, however, the uniqueness holds using the comparison principle in Proposition \ref{compar}. In the same way we reach that $u_n\le u_{n+1}$ for all $n$.

Let $\rho$ be the solution to
\begin{equation}\label{GGGG1}
\left\{
\begin{array}{rcll}
(-\Delta)^{s} \rho &= & 1 & \text{ in }\Omega , \\
\rho & = & 0 & \text{ in }\ren \setminus\Omega.%
\end{array}%
\right.
\end{equation}
From \cite {CDDS}, we know that $\rho\in C^{\a}(\overline{\O})$ where $\a\in(0,\min(2s,1))$. In particular, since $0<s<\frac 12 $, then $\rho\leq C d^{2s}$.

Using $\rho$ as a test function in \eqref{GGGG1} and taking into consideration the previous estimate on $\rho$, it holds
$$
\io u_n dx=\io \frac{u^q_n \rho}{(d(x)+\frac 1n)^{2s}}dx\leq C \io u^q_ndx.
$$
Since $q<1$, by H\"older inequality, we obtain
$$
\io u_n dx\leq C, \mbox{  for all }n.
$$
Thus we get the existence of a measurable function $u$ such that $u_n\uparrow u$ strongly in $L^1(\O)$ as $n\to \infty$. By the Dominated Convergence Theorem, we reach that
$$\dfrac{u^q_n}{(d(x)+\frac 1n)^{2s}}\uparrow \dfrac{u^q}{d^{2s}}\quad \hbox {strongly in }\quad  L^{1}_{loc}(\O).$$
Hence $u$ is, at least, a distributional solution to problem \eqref{PHI}.

Assume that $q<1-2s$, using H\"older inequality we can prove that $\dfrac{u^q}{d^{2s}}\in L^1(\O)$. Thus $u$ is an entropy solution to \eqref{PHI} and then $u\in W^{s,\s}_0(\Omega)$ for all $\s<\frac{N}{N-2s-1}$, see \cite{LPPS} and \cite{AAB}.

{{Let us}} prove now that $u\ge C d^s$. Denote by $\phi_1$ the first positive eigenfunction of the fractional operator, then we know that $\phi_1\backsimeq d^s$. It is not difficult to see that $C\phi_1$ is a subsolution to problem \eqref{PPP1} where $C>0$ can be chosen independently of $n$. Hence by the comparison principle in Lemma \ref{compar}, we conclude that $u_n\ge C\phi_1$. Passing to the limit as $n\to \infty$, we get the desired estimate.

\

\textbf{The second case: $\frac 12\leq s<1$}.
\
Let $\phi$ be the unique solution to problem \eqref{ave1} with $\frac{1}{2}\leq s <\b<1$. Choosing $\phi$ as a test function in \eqref{PPP1} and
taking into consideration that $\phi\backsimeq d^{2s-\beta}$, we reach that
$$
\io \dfrac{u_n}{d^{\b}(x)} dx=\io \frac{u^q_n \phi}{(d(x)+\frac 1n)^{2s}}dx\leq C \io\frac{ u^q_n}{d^{\b}}dx.
$$
Since $\beta<1$, by H\"older inequality, we obtain
$$
\io \dfrac{u_n}{d^{\b}(x)} dx\leq \left(\io \dfrac{u_n}{d^{\b}(x)} dx\right)^{q}\left(\io\frac{1}{ d^\b}dx\right)^{1-q},
$$
Hence
$$
\io \frac{ u_n}{d^{\b}}dx\leq  C.
$$
As consequence, we get the existence of a measurable function $u$ such that $u_n \uparrow u$ in $L^1(\O)$ and $\dfrac{u_n}{(d(x)+\frac 1n)^{\beta}}\uparrow \dfrac{u}{d^{\beta}}$ strongly in $L^{1}(\O).$ It is clear that $u$ solves problem \eqref{PHI}, at least, in the distributional sense.

Notice that, since $u_n=0$ in the set $\ren\backslash \O$, then, in any case, it holds that $u=0$ a.e. in $\ren\backslash \O$.
\end{proof}

\begin{remarks}
In the local case, $s=1$, the authors in \cite{ABDM} proved a strong non existence result to problem \eqref{PHI} for all $q<1$ and as a consequence they get a complete blow-up for the approximating problems. Hence our existence result in Theorem \ref{tt1} exhibits a significative difference between the local and the non local case.
\end{remarks}

\section{The superlinear case: $1<q \le 2^{*}_{s}-1$}\label{sec3}
In this section we are interested to find a positive solution to problem \eqref{PHI} in the superlinear case $1< q \le 2^{*}_{s}-1$.
According to the values of $q$, we will consider two main cases: the subcritical case $q<2^*_{s}-1$ and the critical case $q=2^*_{s}-1$.

\subsection{The Subcritical case: $1<q<2^*_{s}-1$.}\label{ssec00}
The main existence result of this case is the following.
\begin{Theorem}\label{tt2}
Assume that $\frac{1}{2}\leq s<1$ and $1< q <2^{*}_{s}-1$, then problem \eqref{PHI} has a bounded positive solution  $u\in H^{s}_0(\O)\cap L^{\iy}(\O)$.
\end{Theorem}
\begin{proof} We proceed by approximation. Let $u_n\in H^{s}_0(\O)\cap L^{\iy}(\O)$ be the \emph{mountain pass solution } to the approximated problem
\begin{equation}\label{PPP111}
\left\{
\begin{array}{rcll}
(-\Delta)^{s} u_n & = & \dfrac{u^q_n}{(d(x)+\frac 1n)^{2s}} & \text{ in }\Omega , \\
u_n & > & 0 & \text{in }\Omega , \\
u_n & = & 0 & \text{in }\ren \setminus\Omega.%
\end{array}%
\right.
\end{equation}
Then $u_n$ is a critical point of the functional
$$
J_n(v)=\dfrac{a_{N,s}}{2}\iint_{D_{\O}}
\dfrac{|v(x)-v(y)|^{2}}{|x-y|^{N+2s}}\,dxdy-\frac{1}{q+1}\io
\frac{|v(x)|^{q+1}}{(d(x)+\frac 1n)^{2s}}dx.
$$
It is clear that $J_n(u_n)=c_n$, the mountain pass energy level defined by
$$
c_n=\inf_{\g\in \G}\max_{t\in [0,1]}J(\g(t))
$$
where $$\G=\{\g\in \mathcal{C}([0,1],  H^{s}_0(\O))\mbox{  with  }\g(0)=0\mbox{  and  }\g(1)=v_1\in H^{s}_0(\O), J_n(v_1)<0\},$$
with $v_1\in H^s_0(\O)$ is chosen such that
$$J_n(v_1)<<0 \quad \hbox{uniformly in} \quad n.$$
Since
$$
c_n=\dfrac{C(N,s)(q-1)}{2(q+1)}\iint_{D_{\O}}
\dfrac{|u_n(x)-u_n(y)|^{2}}{|x-y|^{N+2s}}\,dxdy,$$
using the fact that
$$0\le c_n\le \max_{t\in [0,\infty)}J_n(tv_1)\le C     \hbox{   for all   } n\in\ene ,$$  we reach that $\{u_n\}_n$ is bounded in $H^{s}_0(\O)$.

\

We claim that
\begin{equation}\label{estim}
||u_n||_{\infty}\le C\qquad \hbox{ for all }\, n.
\end{equation}
Assume by contradiction that there exists a sequence $\{u_n\}\subset H^{s}_0(\O)$ of solutions to \eqref{PPP111} such that $||u_n||_{\infty}=M_n \rightarrow \infty$ as $n \rightarrow \infty$. Let $x_n\in \O$ be such that $u_n(x_n)=M_n$. Since $\{x_n\}_n \subset\O$, a bounded set, we get the existence of $\overline{x}\in \overline{\O}$ such that, up to a subsequence, $x_n\to \overline{x}$.

\textbf{(1) First case : $\overline{x}\in\O$.} We consider the scaled function
$$
v_n(x)=\dfrac{u_n(\mu_n x+x_n)}{M_n}  \hbox{      for    } x_n\in\O,
$$
where $\mu_n=M_n^{\frac{1-q}{2s}}$. Then $v_n$ solves
\begin{equation}\label{PHS0}
\left\{
\begin{array}{rcll}
(-\Delta)^{s} v_n & = & \dfrac{v^q_n}{(d(\mu_n x+x_n)+\frac 1n)^{2s}} & \text{ in }\Omega_n , \\
v_n & > & 0 & \text{in }\Omega_n , \\
v_n & = & 0 & \text{in }\ren \setminus\Omega_n,%
\end{array}%
\right.
\end{equation}
and $\O_n=\dfrac{1}{\mu_n}(\Om -x_n)$.
Clearly, for $x$ fixed, $d(\mu_n x+x_n)+\frac 1n\to d(\overline{x})=C$ as $n\to \infty$.

As in \cite{MMF} (see also \cite{CS}), we can prove that $v_n\in C^{0,\g}$ and $||v_n||_{C^{0,\g}}\leq C$ for some $0<\g<1$.
Passing to the limit as $n\to \infty$, we get the existence of $v\in \mathcal{C}^{0,\g}(\mathbb{R}^{N})\cap L^\infty(\mathbb{R}^{N})$ such that $v(x)\le v(0)=1$ and $v>0$ solves
\begin{equation}\label{lio111}
\left\{
\begin{array}{cc}
(-\D)^s v=Cv^q  & \text{in } \mathbb{R}^{N} , \\
v>0 & \text{in }\mathbb{R}^{N}.%
\end{array}%
\right.
\end{equation}%
Since $q<2^*-1$, we get a contradiction with the non existence result of Theorem \ref{blo}.

\textbf{(2) Second case : $\overline{x}\in\p\O$.}
In this case we set $\mu_n=M_n^{\frac{1-q}{2s}}(d(x_n)+\frac 1n)$, then $v_n$ solves

\begin{equation}\label{PHS1}
\left\{
\begin{array}{rcll}
(-\Delta)^{s} v_n & = & \Big(\dfrac{d(x_n)+\frac 1n}{d(\mu_n x+x_n)+\frac 1n}\Big)^{2s}v_n^{q} & \text{ in }\Omega_n , \\
v_n & > & 0 & \text{ in }\Omega_n , \\
v_n & = & 0 & \text{ in }\ren \setminus\Omega_n.
\end{array}%
\right.
\end{equation}
For $x\in \ren$ fixed, we have $\dfrac{d(x_n)+\frac 1n}{d(\mu_n
x+x_n)+\frac 1n}\to 1$, as $n\to \infty$. Thus, as above, passing to the limit
as $n\to \infty$, we get the existence of $v$ such that either,  $v\in \mathcal{C}^{0,\g}(\mathbb{R}^{N})\cap L^\infty(\mathbb{R}^{N})$ with $v(x))\le v(0)=1$ and $v>0$, solves the problem \eqref{lio111}, or, $v\in \mathcal{C}^{0,\g}(\mathbb{R}^{N}_{+})\cap L^\infty(\mathbb{R}^{N}_{+})$ where $\g\in(0,1)$, $v(x)\le
v(0)=1$ and $v$ solves
\begin{equation*}
\left\{
\begin{array}{rcll}
(-\Delta)^{s} v & = & Cv^{q} & \text{in } \mathbb{R}^{N}_{+} , \\
 v & = & 0 & \text{in }\ren \backslash\mathbb{R}^{N}_{+}.%
\end{array}%
\right.
\end{equation*}%
Since $q<2^*-1$, we reach a contradiction with the non existence results of Theorems \ref{blo} and \ref{blo2}. Hence the claim follows.

\
\

Let us prove now that the sequence $\{u_n\}_n$ is bounded from below, namely that
\begin{equation}\label{gg1}||u_n||_{L^\infty(\O)}\ge \overline{C}>0 \mbox{  for all }n.
\end{equation}
If estimate \eqref{gg1} is false, then we get the existence of a subsequence of $\{u_n\}_n$ denoted also by $\{u_n\}_n$ such that $||u_n||_{L^\infty(\O)}\to 0$ as $n\to \infty$. Hence
$$
  \dfrac{a_{N,s}}{2}\iint_{D_{\O}}
\dfrac{|u_n(x)-u_n(y)|^{2}}{|x-y|^{N+2s}}\,dxdy =\io
\frac{u_n^{q}(x)}{(d(x)+\frac 1n)^{2s}}dx\leq ||u_n||_{L^\infty(\O)}^{q-1} \io \frac{u_n^2(x)}{(d(x)+\frac 1n)^{2s}}dx.
$$
Taking  $n$ large, we obtain that
$||u_n||^{q-1}_{L^\infty(\O)}<<C(\O,N,s)$, the optimal constant in the Hardy inequality stated in \eqref{Ter_Hardy}. Hence we reach a contradiction. Therefore we conclude that
$||u_n||_{L^\infty(\O)}\ge \overline{C}$ for all $n$.

\

Recall that $u_n(x_n)=||u_n||_{L^\infty(\O)}$. We claim that
$$d(x_n)\equiv \text{dist}(x_n,\p\O)>C_1>0,  \quad \hbox{for all} \quad n.$$ We proceed  by contradiction. Assume the existence of a subsequence, $x_n\to \overline{x}\in \p\O$ with $||u_n||_{L^\infty}=u_n(x_n)\to C_2\ge \overline{C}$ as $n\to \infty$. Then as above, we set $$v_n(x,y)=\dfrac{u_n(\mu_nx+x_n)}{M_n}$$ where
$$\mu_n=M_n^{\frac{1-q}{2s}}\Big(d^{2s}(x_n)+\frac 1n\Big)^{\frac {1}{2s}}.$$
Thus, we obtain that  $\mu_n\to 0$ as $n\to \infty$. Following the same Blow-up analysis as above, we reach that $v_n\to v$ strongly in $\mathcal{C}^{0,\g}(\mathbb{R}^{N}_{+})\cap L^\infty(\mathbb{R}^{N}_{+})$  where $v$ solves
\begin{equation*}
\left\{
\begin{array}{rcll}
(-\Delta)^{s} v & = & Cv^{q} & \text{in } \mathbb{R}^{N}_{+} , \\
 v & = & 0 & \text{in }\ren \backslash \mathbb{R}^{N}_{+},
\end{array}%
\right.
\end{equation*}%
which is a contradiction with Theorem \ref{blo2}. Hence the claim follows.
\
\

\
Therefore we conclude that $\{u_n\}_n$ is bounded in $H^{s}_0(\O)\cap L^{\iy}(\O)$. Hence there exists $u\in H^{s}_0(\O)\cap L^{\iy}(\O)$ such that, up to a subsequence, $$u_n\rightharpoonup u \hbox{  weakly in } \,H^{s}_0(\O)\qquad \hbox{and}\quad u_n\to u \hbox{  strongly in } \,L^p(\O) \mbox{  for all }p\ge 1.$$

It is not difficult to show that $u$ solves the problem \eqref{PHI}. To finish, we have just to prove that $u\not\equiv 0$. Assume by contradiction that $u\equiv 0$, then $u_n\to 0$ strongly in $L^p(\O)$ for all $p\ge 1$. We claim that
$$\int_{\O}|(\Delta)^{\frac{s}{2}} u_n|^{2}\phi_1 \to 0\qquad \hbox{as}\quad n\to \infty,$$ where $\phi_1$ is the first eigenfunction to
\begin{equation}\label{ave100}
\left\{
\begin{array}{rcll}
(-\D)^{s} \phi_1 &= & \l_1\phi_1 & \text{ in }\Omega , \\
\phi_1 & = & 0 & \text{  in }\ren\backslash \Omega , \\
\end{array}%
\right.
\end{equation}
To prove the claim, we take $u_n(\phi_1+\frac{c}{n})$ as a test
function in \eqref{PPP111} with  $c\ge
\sup\limits_{\bar{\O}}\dfrac{\phi_1(x)}{d^{s}(x)}$, then

\begin{eqnarray*}
\dyle \int_{\O}(\phi_1+\frac{c}{n}) u_n (-\Delta )^s u_n dx &= &\dyle \int_{\O}(\phi_1+\frac{c}{n}) |(\Delta)^{\frac{s}{2}} u_n|^{2}+\int_{\O}u_n(-\Delta)^{\frac{s}{2}}\phi_1(-\Delta)^{\frac{s}{2}}u_n dx\\ &-& \dyle \frac{a_{N,s}}{2}\int_{\O}u_n(x)\int_{\ren}\frac{(u_n(x)-u_n(y))(\phi_1(x)-\phi_1(y))}{|x-y|^{N+s}}dxdy\\ &\leq &
C\io \dfrac{u_n^{q+1}}{(d+\frac 1n)^s}dx.
\end{eqnarray*}
Thus
\begin{eqnarray*} & \dyle\int_{\O}(\phi_1+\frac{c}{n}) |(\Delta)^{\frac{s}{2}} u_n|^{2}+\frac{\l_1}{2}\int_{\O}u_n^2\phi_1 dx -\dyle \frac{a_{N,s}}{2}\int_{\O}u_n(x)\int_{\ren}\frac{(u_n(x)-u_n(y))(\phi_1(x)-\phi_1(y))}{|x-y|^{N+s}}dxdy\\
 &\le \dyle C\Big(\io \dfrac{u_n^{q+1}}{(d+\frac 1n)^{2s}}dx\Big)^{\frac 12}\Big(\io u_n^{q+1}dx\Big)^{\frac 12}\\
 `&\leq \dyle C\Big(\io u_n^{q+1}dx\Big)^{\frac 12}.
\end{eqnarray*}
Since
$$
\dyle \frac{a_{N,s}}{2}\int_{\O}u_n(x)\int_{\ren}\frac{(u_n(x)-u_n(y))(\phi_1(x)-\phi_1(y))}{|x-y|^{N+s}}dxdy \to 0\mbox{  as   }n\to \infty,
$$
and taking into consideration that
$$
\io u_n^{q+1}dx \to 0\mbox{  as   }n\to \infty,
$$
we reach that $\int_{\O}(\phi_1+\frac{c}{n}) |(\Delta)^{\frac{s}{2}} u_n|^{2}\to 0 \mbox{  as   }n\to \infty$ and the claim follows.

By the elliptic regularity, we conclude that $u_n\to 0$ strongly in $\mathcal{C}^{\g}_{loc}({\O})$. Since $d(x_n)\ge C>0$ for all $n$,
then up to a subsequence, $u_n(x_n)\to 0$ as $n\to \infty$, a contradiction with \eqref{gg1}. Hence $u\gneq 0$ and then the existence result follows.
\end{proof}

\begin{remarks}\label{rem2}
If we consider the problem
\begin{equation}\label{spect2}
\left\{
\begin{array}{cc}
\mathcal{A}^s u=\dfrac{u^q}{d^{2s}} &
\text{in }\Omega , \\
u=0 & \text{on } \p \Omega ,%
\end{array}%
\right.
\end{equation}
where $1<q<2^*-1$, then as in Theorem \ref{tt2}, we are able to show that problem \eqref{spect2} has a nontrivial solution.
This follows using the Caffarelli-Silvestre extension and the Liouville type result obtained in \cite{CCPS}.
\end{remarks}
\subsection{The critical case $q=2_{s}^*-1$.}\label{sub1}
In this section we will consider \eqref{PHI} with $q=2_{s}^*-1$ and $\O=B_R(0)$ is the ball of radius $R$
centered at the origin. We define the space
$$
H^s_{0,rad}(B_R(0))\equiv \{u\in H^s_{0}(B_R(0)):\:u\mbox{  radial  }\}.
$$
Our main existence result is the following.
\begin{Theorem}\label{rad}
Assume that $q=2^*-1$ and that $\O=B_R(0)\subset \ren$ with $N>2s$. Then the problem \eqref{PHI} has a bounded positive radial solution  $u\in  H^s_{0,rad}(B_R(0))$.
\end{Theorem}

Let us define $S(R)$ as follows,
\begin{equation}\label{SSS}
S(R)\equiv\inf_{\phi\in
H^s_{0,rad}(B_R(0))}\dfrac{\dyle \dfrac{a_{N,s}}{2}\int_{\ren}
\int_{\ren}\dfrac{(\phi(x)-\phi(y))^2}{|x-y|^{N+2s}}dxdy}{\Big(\dyle\int_{B_R(0)}
\frac{|\phi|^{2_{s}^*}}{d^{2s}(x)}dx\Big)^{\frac{2}{2_{s}^*}}}.
\end{equation}
In order to prove Theorem \ref{rad}, it suffices to show that $S(R)$ is achieved.

We begin by the next proposition.
\begin{Proposition}\label{DDDp}
We have
\begin{enumerate}
\item $S(R)>0$ for all $R>0$,
\item $S(R)=R^{\frac{4s}{2_{s}^*}}S(1).$
\end{enumerate}
\end{Proposition}
\begin{proof}
The proof follows the same arguments as in \cite{ABDM}. For the reader convenience, we include here some details.
Since $\phi\in H^s_{0,rad}(B_R(0))$, then from \cite{PNI}, it holds that
\begin{equation}\label{radial}
|\phi(x)|\le C|x|^{-\frac{N-2s}{2}}||\phi||_{H^s_{0,rad}(B_R(0))}
\end{equation}
with $C\equiv C(N,s)$. Let $0<R_1<R$, then
$$
\begin{array}{lll}
\dyle\int_0^R\frac{|\phi|^{2^*_s}}{(R-r)^{2s}}r^{N-1}dr &= &\dyle
\int_0^{R_1}\frac{|\phi|^{2^*_s}}{(R-r)^{2s}}r^{N-1}dr+ \dyle
\int_{R_1}^R\frac{|\phi|^{2^*_s}}{(R-r)^{2s}}r^{N-1}dr\\ \\&=&
I_1(R_1)+I_2(R_1).
\end{array}
$$
Using Sobolev inequality, we obtain that
\begin{equation}\label{est00}
I_1(R_1)\le \dyle
\dfrac{1}{(R-R_1)^{2s}}\int_0^{R}|\phi|^{2^*_s}r^{N-1}dr\le
C(R,R_1,N)\,||\phi||_{H^s_{0,rad}(B_R(0))}^{2^*_s}.
\end{equation}
Respect to $I_2$, since $\phi\in H^s_{0,rad}(B_R(0))$, using
\eqref{radial}, we reach that
$$
\begin{array}{lll}
|\phi(x)|^{2^*_s} &= &|\phi(x)|^{2^*_s-2}|\phi(x)|^{2}\\ \\
&\le &C|x|^{-2s}||\phi||^{2^*_s-2}_{H^s_{0,rad}(B_R(0))}|\phi(x)|^{2}\\ \\
&\le & CR^{-2s}_1||\phi||^{2^*_s-2}_{H^s_{0,rad}(B_R(0))}|\phi(x)|^{2}.
\end{array}
$$
Thus, using Hardy inequality,
\begin{equation}\label{est001}
I_2(R)\le
CR^{-2s}_1||\phi||^{2_{s}^*-2}_{H^s_{0,rad}(B_R(0))}\int_{0}^R
\frac{|\phi|^{2}}{(R-r)^{2s}}r^{N-1} dr\le
CR^{-2s}_1||\phi||^{2_{s}^*}_{H^s_{0,rad}(B_R(0))}.
\end{equation}
Therefore, by \eqref{est00} and \eqref{est001} we reach that $S(R)\ge \dfrac{1}{C(N,R,R_1)}>0.$

The second point follows using a rescaling argument.
\end{proof}

We are now able to prove Theorem \ref{rad}.

{\bf Proof of Theorem \ref{rad}.}

Taking into consideration the second point in Proposition \ref{DDDp}, then, we have just to show that $S(R)$ is achieved for some $R>0$.

From the second point in Proposition \ref{DDDp}, we get the existence of $R<1$ such that $S(R)<S$, the Sobolev constant defined in Theorem \ref{Sobolev}. Fix a such $R$ and let $\{u_n\}_n\subset H^s_{0,rad}(B_R(0))$, be a minimizing sequence of $S(R)$ with
$$\dyle\int_0^R \dfrac{|u_n|^{2_{s}^*}}{(R-r)^{2s}}r^{N-1}dr=1.$$

Without loss of generality we can assume that $u_n\ge 0$. Thus
$$||u_n||_{H^s_{0,rad}(B_R(0))}\le C.$$

Hence we get the existence of $u\in H^s_{0,rad}(B_R(0))$ such that
$u_n\rightharpoonup u$ weakly in  $H^s_{0,rad}(B_R(0))$, $u_n\to
u$ strongly in $L^\s(B_R(0))  \,\forall\, \s<2_{s}^*$ and $u_n\to u$
strongly in $L^\s(B_R(0)\backslash B_\e(0))$ for all $\s>1$ and
for all $\e>0$.

We claim that $u\neq 0$ and then $u$ solves \eqref{PHI} with $q=2_{s}^*-1$.

We argue by contradiction. Assume that $u\equiv 0$, then $u_n\to 0$ strongly in
$L^\s(B_R(0)\backslash B_\e(0))$ for all $\s>1$ and for all
$\e>0$. Fix $0<R_1<R$, then
$$
\dyle\frac{u_n^{2^*_s}}{(R-r)^{2s}}r^{N-1}\le
CR^{-2s}_1||u_n||^{2_{s}^*-2}_{H^s_{0,ra}(B_R(0))}
\frac{u_n^{2}}{(R-r)^{2s}}r^{N-1} .
$$

Since by the Hardy inequality it holds that $\dint_{0}^R\dfrac{u_n^{2}}{(R-r)^{2s}}r^{N-1}< \infty$.

Thus by the Dominated Convergence Theorem, we get
$$
\dyle\int_{R_1}^R\frac{|u_n|^{2_{s}^*}}{(R-r)^{2s}}r^{N-1}dr\to 0\mbox{
as   }n\to \infty.
$$
Thus, for all $1<R_1<R$, we have
$$
\dyle\int_{B_{R_1}(0)}\frac{|u_n|^{2_{s}^*}}{(R-|x|)^{2s}}dx\to 1\mbox{
as }n\to \infty.
$$
Hence, in order to show the compactness of the sequence $\{u_n\}_n$  we have to avoid any concentration in zero.

Using Ekeland variational principle, we obtain that, up to a
subsequence,
\begin{equation}\label{SAN}
(-\D)^s u_n=S(R)\frac{u_n^{2_{s}^*-1}}{(R-|x|)^{2s}}+o(1).
\end{equation}
Now, by the concentration compactness principle, see \cite{pallu}, and using the fact that $u_n$ is a radial function, it follows that
\begin{equation}\label{CC1}
|u_n|^{2_{s}^*}\rightharpoonup \nu=\nu_0\d_{0}, \:\:\:|(\D)^{\frac{s}{2}} u_n|^2\rightharpoonup \mu\ge \widetilde{\mu}+\mu_0\d_{0},
\end{equation}
with
\begin{equation}\label{CC2}
S\nu^{\frac{2}{2^*_s}}_0  \leq\mu_0,
\end{equation}
where $\widetilde{\mu}$ is a positive measure with $supp( \widetilde{\mu})\subset \overline{B_R(0)} $.

\noindent For $\e>0$, we consider $\phi_\e\in \mathcal{C}^\infty_0(B_R(0))\cap H^s_{0,rad}(B_R(0))$ such that $0\le \phi_\e\le 1$,
$$ \phi_\e\equiv
1 \quad \hbox{in }\,B_\e(0), \quad \phi_\e\equiv 0\quad \hbox{ in } \,B_{R}(0)\backslash B_{2\e}(0) \mbox{ and }\dfrac{|\phi_\e(x)-\phi_\e(y)|}{|x-y|}\le \frac{C}{\e}.$$
Using $u_n\phi_\e$ as a test function in \eqref{SAN}, it holds that
\begin{equation}\label{concent}
\int_{B_R(0)}\phi_\e u_n (-\Delta )^s u_n dx =S(R)\int_{B_R(0)} \frac{u_n^{2_{s}^* }\phi_\e}{(R-|x|)^{2s}} dx+o(1).
\end{equation}
It is clear that
$$
S(R)\int_{B_R(0)} \frac{u_n^{2_{s}^* }\phi_\e}{(R-|x|)^{2s}} dx\to \nu_0 S(R)\mbox{  as   }n\to \infty \mbox{  and  }\e\to 0.
$$
On the other hand, taking into consideration the properties of operator $(-\D)^s$, we obtain that
\begin{eqnarray*}
\dyle \int_{B_R(0)}\phi_\e u_n (-\Delta )^s u_n dx &= &\dyle \int_{B_R(0)}\phi_\e |(\Delta)^{\frac{s}{2}} u_n|^{2}+\int_{B_R(0)}u_n(-\Delta)^{\frac{s}{2}}\phi_\e(-\Delta)^{\frac{s}{2}}u_n dx\\ &-& \dyle \frac{a_{N,s}}{2}\int_{B_R(0)}u_n(x)\int_{\ren}\frac{(u_n(x)-u_n(y))(\phi_\e(x)-\phi_\e(y))}{|x-y|^{N+s}}dxdy\\ &= &
A_1(\e,n)+A_2(\e,n)+A_3(\e,n).
\end{eqnarray*}
We will estimate each term in the last identity.

Since $supp( \widetilde{\mu})\subset \overline{B_R(0)}$, then using \eqref{CC1}, letting $n\to \infty$ and $\e\to 0$, it holds that
$$
\limit_{\e\to 0}\limit_{n\to \infty}A_1(\e,n)=\limit_{\e\to 0}\limit_{n\to \infty}\int_{B_R(0)} \phi_\e |(\Delta)^{\frac{s}{2}} u_n|^{2}dx \to  \mu \ge\mu_0.
$$
We estimate now the term $A_2(\e,n)$. Recall that $\dfrac{|\phi_\e(x)-\phi_\e(y)|}{|x-y|}\le \dfrac{C}{\e}$, then it follows that
$|(-\Delta)^{\frac{s}{2}}\phi_\e|\le \dfrac{C}{\e}$.

We have
\begin{eqnarray*}
&\dyle\int_{B_R(0)}\left|u_n(-\Delta)^{\frac{s}{2}}\phi_\e(-\Delta)^{\frac{s}{2}}u_n \right|dx=\dyle\int_{B_{2\e}(0)}\left|u_n(-\Delta)^{\frac{s}{2}}\phi_\e(-\Delta)^{\frac{s}{2}}u_n\right|dx\\ &\dyle +\dyle \int_{B_R(0)\backslash B_{2\e}(0)}\left|u_n(-\Delta)^{\frac{s}{2}}\phi_\e(-\Delta)^{\frac{s}{2}}u_n\right|dx=J_1(\e,n)+J_2(\e,n).
\end{eqnarray*}
Since
\begin{eqnarray*}
& \dyle\int_{B_{2\e}(0)}\left|u_n(-\Delta)^{\frac{s}{2}}\phi_\e(-\Delta)^{\frac{s}{2}}u_n dx\right|\le \dyle \frac{C}{\e}\int_{B_{2\e}(0)}\left|u_n(-\Delta)^{\frac{s}{2}}u_n\right|dx,
\end{eqnarray*}
taking into consideration that
$$
\int_{B_{2\e}(0)}\left|u_n(-\Delta)^{\frac{s}{2}}u_n\right|dx\to \int_{B_{2\e}(0)}\left|u(-\Delta)^{\frac{s}{2}}u\right|dx\mbox{  as  }n\to \infty,
$$
$$
\int_{B_{2\e}(0)}\left|u(-\Delta)^{\frac{s}{2}}u\right|dx\le C\e ||u||_{L^{2^*_s}(B_{2\e}(0))},
$$
and since we have assumed that $u=0$, we conclude that $\limit_{\e\to 0}\limit_{n\to \infty}J_1(\e,n)=0.$

\noindent In the same way and using a duality argument we get $\limit_{\e\to 0}\limit_{n\to \infty}J_2(\e,n)=0$.

Combining the above estimates, { {it follows}} that $\limit_{\e\to 0}\limit_{n\to \infty}A_2(\e,n)=0$.

\

We deal now with the last term $A_3(\e,n)$. We have
\begin{eqnarray*}
&\dyle \int_{B_R(0)}u_n(x)\int_{\ren}\frac{(u(x)-u(y))(\phi(x)-\phi(y))}{|x-y|^{N+s}}dxdy=\\
&\dyle \int_{B_R(0)}u_n\int_{B_R(0)}\frac{(u_n(x)-u_n(y))(\phi(x)-\phi(y))}{|x-y|^{N+s}}dxdy + \\ \\ &\dyle \int_{B_R(0)}u_n\int_{\mathcal{C}B_R(0)}\frac{(u_n(x)-u_n(y))(\phi(x)-\phi(y))}{|x-y|^{N+s}}dxdy\\
&\equiv B_1(\e,n)+B_2(\e,n).
\end{eqnarray*}
Respect to $B_1(\e,n)$, we have
\begin{eqnarray*}
&\bigg|\dyle\int_{B_R(0)}u_n(x)\int_{B_R(0)}\frac{(u_n(x)-u_n(y))(\phi(x)-\phi(y))}{|x-y|^{N+s}}dxdy\bigg|\le \\
&\dyle \bigg(\int_{B_R(0)}\int_{B_R(0)}\frac{(u_n(x)-u_n(y))^2}{|x-y|^{N+s}}dxdy\bigg)^{\frac 12}\times
\bigg(\int_{B_R(0)}u^2_n(x)\int_{B_R(0)}\frac{(\phi_\e(x)-\phi_\e(y))^2dy}{|x-y|^{N+s}}dx\bigg)^{\frac 12}\le \\
&\dyle C \bigg(\int_{B_R(0)}u^2_n(x)\int_{B_R(0)}\frac{(\phi_\e(x)-\phi_\e(y))^2dy}{|x-y|^{N+s}}dx\bigg)^{\frac 12},
\end{eqnarray*}
where we have used the fact that the sequence $\{u_n\}_n$ is bounded in  $H^s_{0,rad}(B_R(0))$. Since $B_R(0)\times B_R(0)$ is a bounded domain, then as in the estimate of the term $J_1(\e,n)$, we can show that
$$
\int_{B_R(0)}u^2_n(x)\int_{B_R(0)}\frac{(\phi_\e(x)-\phi_\e(y))^2dy}{|x-y|^{N+s}}dx\to 0\mbox{  as  }n\to \infty \mbox{  and  }\e\to 0.
$$
Respect to $B_2(\e,n)$, since $supp(\phi_\e), supp(u_n) \subset B_R(0)$, using H\"older inequality and taking into consideration that
$B_R(0)\times CB_R(0)\subset D_{B_R(0)}$, it holds that
\begin{eqnarray*}
&\dyle \bigg|\int_{B_R(0)}u_n\int_{\mathcal{C}B_R(0)}\frac{(u_n(x)-u_n(y))(\phi_\e(x)-\phi_\e(y))}{|x-y|^{N+s}}dxdy\bigg|\le\\
&\dyle \bigg(\iint_{D_{B_R(0)}}\frac{(u_n(x)-u_n(y))^2}{|x-y|^{N+s}}dxdy\bigg)^{\frac 12}\times
\bigg(\int_{B_R(0)}u^2_n(x)\phi^2_\e(x)\int_{\mathcal{C}B_R(0)}\frac{dy}{|x-y|^{N+s}}dx\bigg)^{\frac 12}\le \\
&\dyle C\bigg(\int_{B_R(0)}u^2_n(x)\phi^2_\e(x)dx\bigg)^{\frac 12},
\end{eqnarray*}
where in the last estimate, we have used the fact that $\int_{\mathcal{C}B_R(0)}\frac{dy}{|x-y|^{N+s}}dx\le C(R)$.
Hence $\limit_{\e\to 0}\limit_{n\to \infty}B_2(\e,n)=0$.

Therefore, combining the above estimates and passing to the limit in $n$ and $\e$ in \eqref{concent}, we conclude that
$$
\mu_0\le S(R)\nu_0.
$$
Since $ S\nu^{\frac{2}{2^*_s}}_0  \leq\mu_0$, then $S\nu^{\frac{2}{2^*_s}}_0  \leq S(R)\nu_0$. If $\nu_0=0$, then $\mu_0=0$. Hence
$$
\dyle\int_{B_{R}(0)}\frac{|u_n|^{2_s^*}}{(R-|x|)^{2s}}dx\to
\dyle\int_{B_{R}(0)}\frac{|u|^{2_s^*}}{(R-|x|)^{2s}}dx=1
$$
a contradiction with the fact that $u\equiv 0$.

Assume that $\nu_0>0$, then $S\le S(R)\nu^{1-\frac{2}{2^*_s}}_0$. Recall that we have chosen $R<1$ such that
$S(R)\equiv R^{\frac{4s}{2_s^*}}S(1)<S$. In this way we get easily that $\nu_0<1$. Hence  $S\le R^{\frac{4s}{2_s^*}}S(1)$.
Taking into consideration that the Sobolev constant $S$
is independent of the domain, and in particular it is independent
of $R$, then letting $R\to 0$, we reach a contradiction.

Thus $u\neq 0$ and it solves \eqref{PHI} with $q=2_s^*-1$. The strong maximum
principle allows us to get that $u>0$ in $B_R(0)$.

Notice that, from the above computation, we can conclude that
$$
\dyle\int_{B_{R}(0)}\frac{|u_n|^{2_s^*}}{(R-|x|)^2}dx\to
\dyle\int_{B_{R}(0)}\frac{|u|^{2_s^*}}{(R-|x|)^2}dx=1
$$
and then $u$ realize $S(R)$. \cqd

\section{The case : $q<0$}\label{sec4}
In this section, we consider the following problem
\begin{equation}\label{PPP3}
\left\{
\begin{array}{rcll}
(-\Delta)^{s} u &= & \dfrac{f}{u^{\s} d^{\a}(x)} & \text{ in }\Omega , \\
u & > & 0 & \text{ in }\Omega , \\
u & = & 0 & \text{ in }\ren \setminus\Omega,%
\end{array}%
\right.
\end{equation}
where $\s=-q>0$, $\a>0$, $0<s<1$, $d(x)=dist(x,\p\O)$ and $f$ is a nonnegative function under suitable summability conditions that will be specified later.

In the local case, the problem \eqref{PPP3} was treated in \cite{DHR}, see also \cite{BOOR}. In the case where $0<s<1$ and $\a=0$ some existence results were obtained in \cite{BBMI}.

\

In order to study the solvability of problem \eqref{PPP3}, we will analyze the associated approximating problem. Indeed for every $n\in \ene^*$, we consider the following problem
\begin{equation}\label{PPP2}
\left\{
\begin{array}{rcll}
(-\Delta)^{s} u_{n} &= & \dfrac{f_n}{(u_{n}+\frac 1n)^{\s} (d(x)+\frac 1n)^{\a}} & \text{ in }\Omega , \\
u_{n} & > & 0 & \text{ in }\Omega , \\
u_{n}& =& 0 & \text{ in }\ren \setminus\Omega,%
\end{array}%
\right.
\end{equation}
where $f_n:=\min(n,f)$.

Notice that the existence of solution $u_{n}\in H^s_0(\O)$ to \eqref{PPP2} follows using the Schauder fixed point theorem. Obviously $u_{n}\in L^{\infty}(\Omega)$.

We start {{by}} proving the next result.
\begin{Lemma}\label{inc}
The sequence $\{u_{n}\}_n$ of the solutions to problem \eqref{PPP2} is increasing in $n$  and for every $\widetilde{\O}\subset\subset \O$, there exists a positive constant $C(\O)$ independent of $n$, such that
$$
u_{n}\geq C(\widetilde{\O})>0.
$$
\end{Lemma}
\begin{proof}
Fixed $n\in \ene$, then by subtracting, it holds that
$$
\begin{array}{lll}
\dyle(-\D)^s( u_{n}-u_{n+1}) & = & \dfrac{f_n}{(u_{n}+\frac 1n)^{\s} (d(x)+\frac 1n)^{\a}}-\dfrac{f_{n+1}}{(u_{n+1}+\frac 1{n+1})^{\s} (d(x)+\frac {1}{n+1})^{\a}}\dyle \\
 & \le & \dfrac{f_{n+1}}{(u_{n}+\frac 1{n+1})^{\s} (d(x)+\frac {1}{n+1})^{\a}}-\dfrac{f_{n+1}}{(u_{n+1}+\frac 1{n+1})^{\s} (d(x)+\frac {1}{n+1})^{\a}}\dyle \\
& = & \dfrac{f_{n+1}}{(d(x)+\frac {1}{n+1})^{\a}}\left(\dfrac{(u_{n+1}+\frac 1{n+1})^{\s}-(u_{n}+\frac 1{n+1})^{\s}} {(u_{n+1}+\frac 1{n+1})^{\s}(u_{n}+\frac 1{n+1})^{\s}}\right)\dyle .\\
\end{array}
$$
Using $( u_{n}-u_{n+1})_{+}$ as a test function, we obtain that
$$
\int_{\ren}\dfrac{f_{n+1}}{(d(x)+\frac {1}{n+1})^{\a}}\left(\dfrac{(u_{n+1}+\frac 1{n+1})^{\s}-(u_{n}+\frac 1{n+1})^{\s}} {(u_{n+1}+\frac 1{n+1})^{\s}(u_{n}+\frac 1{n+1})^{\s}}\right)( u_{n}-u_{n+1})_{+}dx\leq 0.
$$
Since
$$
\int_{\ren} (-\D)^s( u_{n}-u_{n+1})( u_{n}-u_{n+1})_{+}\ge \iint_{D_\O}\dfrac{( u_{n}-u_{n+1})_{+}^2}{|x-y|^{N+2s}}dxdy,
$$
we conclude that $(u_{n}-u_{n+1})_{+}=0$, and then $u_{n}\le u_{n+1}$, for all $n$.
On the other hand, we know that $u_{1}\in L^{\iy}(\O)$ and
$$
(-\D)^s u_{1}=\dfrac{f_1}{(u_{1}+1)^{\s} (d(x)+1)^{\a}}\ge \dfrac{f_1}{(||u_{1}||_\infty+1)^{\s}}\ge \dfrac{f_1}{(||u_{1}||_\infty+1)^{\s}}.
$$
Thus by the strong Maximum principle $u_{1}>0$ in $\O$. Hence for every $\widetilde{\O}\subset\subset \O$, there exists a positive constant $C(\tilde{\O})$ independent of $n$ such that
$$
u_{n}\geq u_{1}\ge C(\widetilde{\O})>0.
$$

\end{proof}

\begin{remarks} As a conclusion of the above computation, we obtain that $u_{n}$ is the unique positive solution to problem \eqref{PPP2}.
\end{remarks}
Let us begin by analyzing the case $\a<1$. We have the next existence result.
\begin{Theorem}\label{exist1}
Assume that $\a<1$ and $f\in L^{p}(\O)$ with $p'\a<1$, then problem \eqref{PPP3} has a distributional solution $u$ such that $u^{\frac{\s+1}{2}}\in H^s_0(\O)$.
\end{Theorem}

\begin{proof}
Recall that $u_{n}$ is the unique solution to problem \eqref{PPP2}. Using $u^\s_n$ as a test function in \eqref{PPP2} it follows that
$$
\iint_{\re^{2N}}\frac{(u_n(x)-u_n(y))(u^{\s}_n(x)-u^{\s}_n(y))}{|x-y|^{N+2s}}dxdy\le \io \frac{f_n}{(d(x)+\frac 1n)^{\a}}dx.
$$
Notice that, for all $(a, b) \in (\re^+)^2$ and for all $\s>0$, we have
\begin{equation}\label{alge3}
(a-b)(a^{\s}-b^{\s})\ge c_3|a^{\frac{\s+1}{2}}-b^{\frac{\s+1}{2}}|^2,
\end{equation}
hence using H\"older inequality, it follows that
$$
\iint_{\re^{2N}}\frac{(u^{\frac{\s+1}{2}}_n(x)-u^{\frac{\s+1}{2}}_n(y))^2}{|x-y|^{N+2s}}dxdy\le ||f||_{L^p(\O)}\bigg(\io \frac{1}{d^{p'\a}(x)}dx\bigg)^{\frac{1}{p'}}.
$$
Since $p'\a<1$, then $||u_n^{\frac{\s+1}{2}}||_{H^s_0(\O)}\le C$. Hence we get the existence of a measurable function $u$ such that $u_n\uparrow u$ a.e in $\O$, $u^{\frac{\s+1}{2}}\in H^s_0(\O)$ and $u_n^{\frac{\s+1}{2}}\rightharpoonup u^{\frac{\s+1}{2}}$ weakly in $H^s_0(\O)$. It is not difficult to show that $u$ is a distributional solution to problem \eqref{PPP3}.
\end{proof}

In the case of general datum $f$, we have the next existence result.
\begin{Theorem}\label{Abdl}
Assume $\a, \s>0$, then we have:
\begin{enumerate}
\item If $ f\in L^{1}(\O, d^{s-\a})$ , then  for all $\s>0$, there exists a positive solution $u$ to problem \eqref{PPP3} in the sense of Definition \ref{avee} such that $\frac{u^{\s+1}}{d^\beta}\in L^{1}(\O)$ for all $0<\b<s$.
\item If  $ f\in L^{1}(\O, d^{s-\a} \log(\frac{D}{d}))$, where $D>\text{dist}(x,\p\O)$, then for all $ \s>0$, there exists a positive solution $u$  to problem \eqref{PPP3} in the sense of Definition \ref{avee} such that $\frac{u^{\s+1}}{d^s}\in L^{1}(\O)$.
\item If  $ f\in L^{1}(\O, d^{2s-\b-\a} )$ for some $\beta\in (s,2s)$, then for all $ \s>0$, there exists a positive solution $u$ to problem \eqref{PPP3}  in the sense of the Definition \ref{avee} such that $\frac{u^{\s+1}}{d^\beta}\in L^{1}(\O)$.
\end{enumerate}
\end{Theorem}
\begin{proof} We proceed by iteration. Consider $u_n$ to be the unique positive solution to the problem \eqref{PPP2}, then using the Kato inequality it holds that
$$
(-\Delta)^s u^{\s+1}_n\le (\s+1)u^\s_n(-\Delta)^s u_n.
$$
Define $\phi$ as the unique positive solution to \eqref{ave1} with $\beta\in (0,s+1)$. Using $\phi$ as a test function in the previous inequality, it holds,
\begin{eqnarray*}
\io \phi(-\Delta)^s u^{\s+1}_ndx &\leq& (\s+1)\io  \phi u^\s_n(-\Delta)^su_n dx \\
& = &\io \dfrac{u^\s_n \phi f_n}{(u_{n}+\frac 1n)^{\s} (d(x)+\frac 1n)^{\a}} dx\\
& \leq & \io \dfrac{\phi f}{d^\a(x)} dx.
\end{eqnarray*}
Suppose that $f$ satisfies the first condition in Theorem \ref{Abdl}, then choosing $\beta\in (0,s)$ in Theorem \ref{avee}, we know that $\phi\backsimeq d^s$. Hence taking into consideration that $f\in L^{1}(\O, d^{s-\a})$, it follows that
$$
\io \dfrac{u^{\s+1}_n}{d^\b}dx\leq C\io f(x) d^{s-\a}(x)\,dx<\iy.
$$
Since $\{u_n\}_n$ is a monotone sequence in $n$, we get the existence of a measurable function  $u$ such that,
$$
\dfrac{u^{\s+1}_n}{d^{\b}}\rightarrow \dfrac{u^{\s+1}}{d^{\b}},  \quad n\rightarrow \iy.
$$
It is clear that $u=0 \text{ in }\ren \setminus\Omega$, hence $u$ is a positive solution to problem \eqref{PPP3} in the sense of Definition \ref{distri}.

\

For the second (resp. the third) case, it suffices to take $\beta=s$ (resp. $\beta\in (s,2s)$) and to use the fact that $v \backsimeq d^s (x)\log(\frac{D}{d})$(resp. $v \backsimeq d^{2s-\b} (x)$). Hence following closely the same calculation as in first case and passing to the limit, we reach the existence result.
\end{proof}

\

{\bf Acknowledgments.} The authors would like to express their gratitude to the anonymous reviewers for their comments and suggestions that improve the last
version of the manuscript.

\end{document}